\documentclass[amstex,12pt, amssymb]{article}

\usepackage{mathtext}
\usepackage[cp1251]{inputenc}
\usepackage[T2A]{fontenc}
\usepackage[dvips]{graphicx}
\usepackage{amsmath}
\usepackage{amssymb}
\usepackage{amsxtra}
\usepackage{latexsym}
\usepackage{ifthen}

\newcounter{lemma}[section]

\newcounter{corollary}[section]

\newcounter{remark}[section]

\newcounter{theorem}[section]

\newcounter{proposition}[section]

\newcounter{example}

\numberwithin{equation}{section}

\begin{document}

\markboth{E~.SEVOST'YANOV, V.~TARGONSKII,
N.~ILKEVYCH}{\centerline{ON KOEBE'S THEOREM ...}}

\def\cc{\setcounter{equation}{0}
\setcounter{figure}{0}\setcounter{table}{0}}

\overfullrule=0pt


\author{EVGENY SEVOST'YANOV, VALERY TARGONSKII, NATALIYA ILKEVYCH}

\title{
{\bf ON KOEBE'S THEOREM FOR MAPPINGS WITH INTEGRAL CONSTRAINTS}}

\date{\today}
\maketitle

\begin{abstract}
We study mappings that satisfy the inverse modulus inequality of
Poletsky type with respect to $p$-modulus. Given $n-1<p\leqslant n,$
we show that, the image of some ball contains a fixed ball under
mappings mentioned above. This statement can be interpreted as the
well-known analogue of Koebe's theorem for analytic functions. As a
consequence, we obtain the openness and discreteness of the limit
mapping in the class under study. The paper also studies mappings of
the Orlicz-Sobolev classes, for which an analogue of the Koebe
one-quarter theorem is obtained as a consequence of the main results
\end{abstract}



\section{Introduction}

The classical Koebe theorem for analytic functions states that
conformal mappings of the unit disk with hydrodynamic normalization
at the origin cover a disk centered at the origin of radius 1/4. Let
us directly recall the formulation of this theorem,
see~\cite[Theorem~1.3]{CG}.

\medskip
{\bf Theorem A.} {\it Let $f:{\Bbb D}\rightarrow {\Bbb C}$ be an
univalent analytic function such that $f(0)=0$ and
$f^{\,\prime}(0)=1.$ Then the image of $f$ covers the open disk
centered at $0$ of radius one-quarter, that is, $f({\Bbb D})\supset
B(0, 1/4).$}

\medskip
Many authors have obtained some other versions of the analogue of
Koebe's theorem, including results for wider classes of mappings,
see, e.g.,  \cite{AFW}, \cite{M} and \cite{Ra}. In particular, we
have recently obtained results on this topic, see, e.g., \cite{ST}.
Also in~\cite{Cr}, some results on the convergence of mappings were
obtained, which essentially used Koebe's theorem, established for
more general classes of mappings compared to~\cite{ST}. As for the
present paper, we consider here classes of mappings satisfying the
inverse Poletsky inequality, in which the majorant $Q$ may depend on
the mapping. The so-called integral constraints are imposed on these
majorants, which control them by means of a certain convex
increasing function~$\Phi,$ see, e.g., \cite{RS}. The main result of
the paper is an analogue of Koebe's theorem for such mappings. As a
consequence, we also establish the openness and discreteness of the
limit mapping using the approaches we have developed, as well as the
methodology of the paper~\cite{Cr}. One of the goals of the paper is
also to obtain corresponding results for Orlicz-Sobolev classes. We
establish these results using the connection between Orlicz-Sobolev
classes and mappings with the inverse Poletsky inequality.

\medskip
Below $dm(x)$ denotes the element of the Lebesgue measure in ${\Bbb
R}^n.$ Everywhere further the boundary $\partial A $ of the set $A$
and the closure $\overline{A}$ should be understood in the sense of
the extended Euclidean space $\overline{{\Bbb R}^n}.$ Recall that, a
Borel function $\rho:{\Bbb R}^n\,\rightarrow [0,\infty] $ is called
{\it admissible} for the family $\Gamma$ of paths $\gamma$ in ${\Bbb
R}^n,$ if the relation
$\int\limits_{\gamma}\rho (x)\, |dx|\geqslant 1$
holds for all (locally rectifiable) paths $ \gamma \in \Gamma.$ In
this case, we write: $\rho \in {\rm adm} \,\Gamma .$ Given
$p\geqslant 1,$ {\it $p$-modulus} of $\Gamma $ is defined by the
equality
$
M_p(\Gamma)=\inf\limits_{\rho \in \,{\rm adm}\,\Gamma}
\int\limits_{{\Bbb R}^n} \rho^p (x)\,dm(x)\,.
$
Let $y_0\in {\Bbb R}^n,$ $0<r_1<r_2<\infty$ and
\begin{equation*}\label{eq1**}
A=A(y_0, r_1,r_2)=\left\{ y\,\in\,{\Bbb R}^n:
r_1<|y-y_0|<r_2\right\}\,.\end{equation*}
Given $x_0\in{\Bbb R}^n,$ we put
$B(x_0, r)=\{x\in {\Bbb R}^n: |x-x_0|<r\}\,, \quad {\Bbb B}^n=B(0,
1)\,,$
$S(x_0,r) = \{ x\,\in\,{\Bbb R}^n : |x-x_0|=r\}\,.$
A mapping $f: D \rightarrow{\Bbb R}^n$ is called {\it discrete} if
the pre-image $\{f^{-1}\left(y\right)\}$ of any point $y\,\in\,{\Bbb
R}^n$ consists of isolated points, and {\it open} if the image of
any open set $U\subset D$ is an open set in ${\Bbb R}^n.$

Given sets $E,$ $F\subset\overline{{\Bbb R}^n}$ and a domain
$D\subset {\Bbb R}^n$ we denote by $\Gamma(E,F,D)$ the family of all
paths $\gamma:[a,b]\rightarrow \overline{{\Bbb R}^n}$ such that
$\gamma(a)\in E,\gamma(b)\in\,F$ and $\gamma(t)\in D$ for $t \in (a,
b).$ Given a mapping $f:D\rightarrow {\Bbb R}^n,$ a point $y_0\in
\overline{f(D)}\setminus\{\infty\},$ and $0<r_1<r_2<r_0<\infty,$ we
denote by $\Gamma_f(y_0, r_1, r_2)$ a family of all paths $\gamma$
in $D$ such that $f(\gamma)\in \Gamma(S(y_0, r_1), S(y_0, r_2),
A(y_0,r_1,r_2)).$ Let $Q:{\Bbb R}^n\rightarrow [0, \infty]$ be a
Lebesgue measurable function and let $p\geqslant 1.$ We say that
{\it $f$ satisfies the inverse Poletsky inequality at a point
$y_0\in \overline{f(D)}\setminus\{\infty\}$ with respect to
$p$-modulus,} if the relation
\begin{equation}\label{eq2*A}
M_p(\Gamma_f(y_0, r_1, r_2))\leqslant
\int\limits_{A(y_0,r_1,r_2)\cap f(D)} Q(y)\cdot
\eta^{\,p}(|y-y_0|)\, dm(y)
\end{equation}
holds for any $0<r_1<r_2<r_0<\infty$ and any Lebesgue measurable
function $\eta: (r_1,r_2)\rightarrow [0,\infty ]$ such that
\begin{equation}\label{eqA2}
\int\limits_{r_1}^{r_2}\eta(r)\, dr\geqslant 1\,.
\end{equation}
The definition of the relation~(\ref{eq2*A}) at the point
$y_0=\infty$ may be given by the using of the inversion
$\psi(y)=\frac{y}{|y|^2}$ at the origin. This means that $f$
satisfies the relation~(\ref{eq2*A}) for $y_0=0$ and
$\widetilde{Q}(y):=Q\left(\frac{y}{|y|^2}\right)$ instead $Q(y).$

\medskip
Quite a lot is known about the role of
relations~(\ref{eq2*A})--(\ref{eqA2}) in mapping theory, as well as
about the fulfillment of these conditions in specific classes, see,
e.g., \cite{MRSY}, \cite{Sev$_1$}, \cite{ST} and \cite{Va}. We also
indicate the application of these inequalities to the Orlicz–Sobolev
classes, see the last section of the manuscript.

\medskip
Let $h$ be a chordal metric in $\overline{{\Bbb R}^n},$
$h(x,\infty)=\frac{1}{\sqrt{1+{|x|}^2}},$
\begin{equation*}\label{eq3C}
h(x,y)=\frac{|x-y|}{\sqrt{1+{|x|}^2} \sqrt{1+{|y|}^2}}\qquad x\ne
\infty\ne y\,.
\end{equation*}
and let $h(E):=\sup\limits_{x,y\in E}\,h(x,y)$ be a chordal diameter
of a set~$E\subset \overline{{\Bbb R}^n}$ (see, e.g.,
\cite[Definition~12.1]{Va}).

\medskip
Given $p\geqslant 1,$ a non-decreasing function
$\Phi\colon\overline{{\Bbb R}^{+}}\rightarrow\overline{{\Bbb
R}^{+}},$ $a, b\in D,$ $a\ne b,$ $\delta>0$ we denote by
$\frak{F}^{\Phi, p}_{a, b, \delta}(D)$ the family of all open
discrete mappings $f:D\rightarrow {\Bbb R}^n,$ $n\geqslant 2,$ for
which there exists a Lebesgue measurable function $Q=Q_f:{\Bbb
R}^n\rightarrow [0, \infty]$ satisfying
relations~(\ref{eq2*A})--(\ref{eqA2}) at any point $y_0\in
\overline{{\Bbb R}^n}$ with
\begin{equation}\label{e3.3.1}
\int\limits_{{\Bbb
R}^n}\Phi(Q_f(y))\cdot\frac{dm(y)}{(1+|y|^2)^n}<\infty
\end{equation}
such that $h(f(a), f(b))\geqslant \delta.$ The following statement
holds.

\medskip
\begin{theorem}\label{th1}
{\it Let $n\geqslant 2,$ $p\in (n-1, n]$ and let $D$ be a domain in
${\Bbb R}^n.$ Let also $\Phi:\overline{{\Bbb R^{+}}}\rightarrow
\overline{{\Bbb R^{+}}}$ be an increasing convex function that
satisfies the condition
\begin{equation}\label{eq2A}
\int\limits_{\delta}^{\infty}\frac{d\tau}{\tau(\Phi^{\,-1}(\tau))^{\frac{1}{p-1}}}=\infty
\end{equation}
for some~$\delta>\Phi(0).$ Assume that, the family $\frak{F}^{Q,
p}_{a, b, \delta}(D)$ is equicontinuous at $a$ and $b.$ Then for
every compactum $K$ in $D$ and for every $0<\varepsilon<{\rm
dist}\,(K,
\partial D)$ there exists
$r_0=r_0(\varepsilon, K)>0$ which does not depend on $f,$ such that
$f(B(x_0, \varepsilon))\supset B_h(f(x_0), r_0)$
for all $f\in\frak{F}^{\Phi, p}_{a, b, \delta}(D)$ and all $x_0\in
K,$ where $B_h(f(x_0), r_0)=\{w\in \overline{{\Bbb R}^n}: h(w,
f(x_0))<r_0\}.$}
\end{theorem}
\medskip
\begin{remark}\label{rem1}
Assume that, $Q$ satisfies~(\ref{e3.3.1}) with some
$\Phi\colon\overline{{\Bbb R}^{+}}\rightarrow\overline{{\Bbb
R}^{+}}.$ Set
$$\widetilde{Q}(y)=\begin{cases}Q(y), & Q(y)\geqslant 1\\ 1, & Q(y)<1\end{cases}\,.$$
Observe that $\widetilde{Q}(y)$ satisfies (\ref{e3.3.1}). Indeed,
\begin{gather*}
\int\limits_D\Phi(\widetilde{Q}(y))\frac{dm(y)}{\left(1+|y|^2\right)^n}=
\int\limits_{\{y\in D: Q(y)< 1
\}}\Phi(\widetilde{Q}(y))\frac{dm(y)}{\left(1+|y|^2\right)^n}\\+
\int\limits_{\{y\in D: Q(y)\geqslant 1\}
}\Phi(\widetilde{Q}(y))\frac{dm(y)}{\left(1+|y|^2\right)^n}\leqslant
M_0+\Phi(1)\int\limits_{{\Bbb
R}^n}\frac{dm(y)}{\left(1+|y|^2\right)^n}=M^{\,\prime}_0<\infty\,.
\end{gather*}
\end{remark}

\begin{remark}
If $p=n$ and $Q\in L^1({\Bbb R}^n),$ then the family $\frak{F}^{Q,
p}_{a, b, \delta}(D)$ is equicontinuous. Indeed, let $\frak{F}_Q(D)$
be a family of all open discrete mappings in $D$ such that the
relations~(\ref{eq2*A})--(\ref{eqA2}) holds for any $y_0\in f(D)$
with $p=n.$ If $Q\in L^1({\Bbb R}^n),$ then for any $x_0\in D$ and
any $r_0>0$ such that $0<r_0<{\rm dist}(x_0,
\partial D)$ the inequality
$
|f(x)-f(x_0)|\leqslant\frac{C_n\cdot (\Vert
Q\Vert_1)^{1/n}}{\log^{1/n}\left(1+\frac{r_0}{2|x-x_0|}\right)}
$
holds for any $x, y\in B(x_0, r_0)$ and $f\in \frak{F}_Q(D),$ where
$\Vert Q\Vert_1$ denotes the $L^1$-norm of $Q$ in ${\Bbb R}^n$ and
$C_n>0$ is some constant depending only on $n.$ In particular,
$\frak{F}_Q(D)$ is equicontinuous in $D,$ see, e.g.,
\cite[Theorem~1.1]{SSD}. Thus, Theorem~\ref{th1} holds for
$\frak{F}^{Q, n}_{a, b, \delta}(D)$ under the condition $Q\in
L^1({\Bbb R}^n),$ see e.g.~\cite{ST}.
\end{remark}

\medskip
Note that the above analogue of Koebe's theorem has an important
application in the field of convergence of mappings. In this
connection, we first recall the classical result on the convergence
of quasiregular mappings (see, for example,
\cite[Theorem~9.2.II]{Re}).

\medskip
{\bf Theorem B.} Let $f_j:D\rightarrow {\Bbb R}^n,$ $n\geqslant 2,$
$j=1,2,\ldots,$ be a sequence of $K$-quasiregular mappings
converging to some mapping $f:D\rightarrow {\Bbb R}^n$ as
$j\rightarrow\infty$ locally uniformly in $D.$ Then either $f$ is
$K$-quasiregular, of $f$ is a constant. In particular, in the first
case $f$ is discrete and open (see \cite[Theorems~6.3.II
and~6.4.II]{Re}).

\medskip
As for the classes we are studying in~(\ref{eq2*A})--(\ref{eqA2}),
we have the following analogue of Theorem~B,
cf.~\cite[Theorem~1.1]{Cr}.

\medskip
\begin{theorem}\label{th2}
{\it\, Let $D$ be a domain in ${\Bbb R}^n,$ $n\geqslant 2,$ let
$p>n-1$ and let $\Phi\colon\overline{{\Bbb
R}^{+}}\rightarrow\overline{{\Bbb R}^{+}}$ be a non-decreasing
function. Let $f_j:D\rightarrow {\Bbb R}^n,$ $n\geqslant 2,$
$j=1,2,\ldots,$ be a sequence of open discrete mappings satisfying
the conditions~(\ref{eq2*A})--(\ref{eqA2}) with some function
$Q=Q_j:{\Bbb R}^n\rightarrow[0, \infty]$ at any point $y_0\in
\overline{{\Bbb R}^n}$ for which~(\ref{e3.3.1}) holds with
$Q_{f_j}:=Q_j$ and converging to some mapping $f:D\rightarrow {\Bbb
R}^n$ as $j\rightarrow\infty$ locally uniformly in $D.$ Assume that
the condition~(\ref{eq2A}) holds. Then either $f$ is a constant, or
$f$ is open and discrete.}
\end{theorem}

\medskip
\begin{remark}\label{rem2}
For the case $p=n$ and some another conditions on the function~$Q,$
Theorem~\ref{th2} was established in~\cite{ST}, except for the
assertion about the discreteness of the limit mapping. In~\cite{Cr},
the mentioned discreteness was proved even for the case of a
non-conformal modulus. Our contribution to this result is that we do
not require the family of mappings in~(\ref{eq2*A})--(\ref{eqA2}) to
have a fixed majorant $Q.$ The proof of openness is similar
to~\cite{ST}, the methodology of proving discreteness is
from~\cite{Cr}.
\end{remark}

\section{Preliminaries}

Let $(X, \mu)$ be a metric space with measure $\mu.$ For each real
number $n\ge 1,$ we define {\it the Loewner function} $\phi_n:(0,
\infty)\rightarrow [0, \infty)$ on $X$ as
$\phi_n(t)=\inf\{M_n(\Gamma(E, F, X)): \Delta(E, F)\leqslant t\},$
where the infimum is taken over all disjoint nondegenerate continua
$E$ and $F$ in $X$ and
$\Delta(E, F):=\frac{{\rm dist}\,(E, F)}{\min\{{\rm diam\,}E, {\rm
diam\,}F\}}.$
A pathwise connected metric measure space $(X, \mu)$ is said to be a
{\it Loewner space} of exponent $n,$ or an $n$-Loewner space, if the
Loewner function $\phi_n(t)$ is positive for all $t> 0$ (see
\cite[Section~2.5]{MRSY} or \cite[Ch.~8]{He}). Observe that ${\Bbb
R}^n$ and ${\Bbb B}^n\subset {\Bbb R}^n$ are Loewner spaces (see
\cite[Theorem~8.2 and Example~8.24(a)]{He}). As known, the condition
$\mu(B(x_0, r))\geqslant C\cdot r^n$ holds in Loewner spaces $X$ for
a constant $C>0$, every point $x_0\in X$  and all $r<{\rm diam}\,X.$
The following definition can be found in \cite[section~1.4,
ch.~I]{He} or \cite[section~1]{AS}. A measure $\mu$ in a metric
space is called doubling if all balls have finite and positive
measure and there is a constant $C\geqslant 1$ such that $\mu(B(x_0,
2r))\le C\cdot \mu(B(x_0, r))$ for every $x_0\in X$ and all $r>0$.
We also call a metric measure space $(X, \mu)$ {\it doubling} if
$\mu$ is a doubling measure. A~metric space $(X,d,\mu)$ is called
{\it\,$\widetilde{Q}$-Ahlfors-regular} for some
$\widetilde{Q}\geqslant 1$ if, for any $x_0\in X$ and some constant
$C\geqslant 1$,
$
\frac{1}{C}R^{\widetilde{Q}}\leqslant \mu(B(x_0, R))\leqslant
CR^{\widetilde{Q}}\,.
$
As is well known, Ahlfors $\alpha$-regular spaces have Hausdorff
dimension $\alpha$ (see \cite[p.~61--62]{He}). Let $(X,d,\mu)$ be
a~metric measure space with metric~$d$ and a~locally finite Borel
measure~$\mu$. Following~\cite{He}, \S\,7.22, a~Borel function
$\rho\colon X\rightarrow [0,\infty]$ is said to be an {\it\,upper
gradient} of a~function $u\colon X\rightarrow {\Bbb R}$ if $
|u(x)-u(y)|\leqslant \int\limits_{\gamma}\rho\,|dx| $ for any
rectifiable path~$\gamma$ connecting the points $x$ and $y\in X$,
where, as usual, $\displaystyle\int\limits_{\gamma}\rho\,ds$~denotes
the linear integral of the~function~$\rho$ over the~path~$\gamma$.
We say that a~space~$X$ admits the $(1;p)$-Poincar\'e inequality if
there exist constants $C\geqslant 1$ and $\tau>0$ such that
$$
\frac{1}{\mu(B)}\int\limits_{B}|u-u_B|\,d\mu(x)\leqslant
C(\operatorname{diam}B) \biggl(\frac{1}{\mu(\tau
B)}\int\limits_{\tau B}\rho^p\,d\mu(x)\biggr)^{1/p}
$$
for any ball $B\subset X$ and arbitrary locally bounded continuous
function $u\colon X\rightarrow {\Bbb R}$ and any upper
gradient~$\rho$ of~$u$, where
$
u_B:=\frac{1}{\mu(B)}\int\limits_{B}u\,d\mu(x).
$
The following result holds (see \cite[Proposition~4.7]{AS}).

\begin{proposition}\label{pr_2}
Let $X$ be a $Q$-Ahlfors regular metric measure space that supports
$(1; p)$-Poincar\'{e} inequality for some $p>1$ such that
$Q-1<p\leqslant Q.$ Then there exists a constant $M>0$ having the
property that, for $x\in X,$ $R>0$ and continua $E$ and $F$ in $B(x,
R),$
$M_p(\Gamma(E, F, X))\geqslant \frac{1}{M} \cdot\frac{\min\{{\rm
diam}\,E, {\rm diam}\,F\}}{R^{1+p-Q}}\,.$
\end{proposition}

\medskip
The following statement holds.

\medskip
\begin{lemma}\label{lem3}
{\it\, Let $n\geqslant 2,$ let $p\in (n-1, n],$ let $x_0\in {\Bbb
R}^n,$ let $\varepsilon_1>0,$ and let $A$ be a (non-degenerate)
continuum in $B(x_0, \varepsilon_1)\subset {\Bbb R}^n.$ Let $r>0$
and let $C_j,$ $j=1,2,\ldots, $ be a sequence of continua in $B(x_0,
\varepsilon_1)$ such that $h(C_j)\geqslant r,$
$h(C_j)=\sup\limits_{x, y\in C_j}h(x, y).$ Then there is $R_0>0$
such that $M_p(\Gamma(C_j, A, D))\geqslant R_0\quad\forall\,\,j\in
{\Bbb N}\,.$}
\end{lemma}

\begin{proof}
By comments given above, the unit ball ${\Bbb B}^n$ is Ahlfors
$n$-regular. Now, the ball $B(x_0, \varepsilon_1)$ is also is
Ahlfors $n$-regular. In addition, $(1;p)$-Poincar\'e inequality
holds for any $p\geqslant 1$ (see \cite[Theorem~10.5]{HK}). Now, the
desired statement follows by Proposition~\ref{pr_2}.~$\Box$
\end{proof}

\medskip
Let $E_0,$ $E_1$ be sets in $D\subset {\Bbb R}^n$. The following
estimate holds (see \cite[Theorem~4]{Car}).

\medskip
\begin{proposition}\label{pr1}
{\it Let $A(0, a, b)=\{a<|x|<b\}$ be a ring containing in $D\subset
{\Bbb R}^n$ such that $S(0, r)$ intersects $E_0$ and $E_1$ for any
$r\in (a,b)$ where $E_0\cap E_1=\varnothing.$ Then $M_p(\Gamma(E_0,
E_1, D))\geqslant \frac{2^nb_{n,p}}{n-p}(b^{n-p}-a^{n-p})$ for any
$p\in(n-1, n),$ where $b_{n,p}$ is a constant depending only $n$ and
$p$.}
\end{proposition}

\medskip
The version of the following lemma is established in the case $p=n$
in \cite[Theorem~3.1]{Na$_2$}. For the case $p\ne n$ its proof is
similar and, therefore, is omitted.

\medskip
\begin{lemma}\label{lem5}
{\it\, Let $p>n-1,$ $F_1,$ $F_2,$ $F_3$ be three sets in a domain
$D$ and let $\Gamma_{i, j}=\Gamma(F_i, F_j, D),$ $1\leqslant i,
j\leqslant 3.$  Then
\begin{equation*}\label{eq32***}
M_p(\Gamma_{1, 2})\geqslant 3^{-p}\min\{M_p(\Gamma_{1, 2}),
M_p(\Gamma_{2, 3}), \inf M_p(\Gamma(|\gamma_{1, 3}|, |\gamma_{2,
3}|, D))\}\,,
\end{equation*}
where the infimum is taken over all rectifiable paths $\gamma_{1,
3}\in \Gamma_{1, 3}$ and $\gamma_{2, 3}\in \Gamma_{2, 3}.$ }
\end{lemma}

\medskip
A version of the following lemma is established in the case $p=n$ in
\cite[Theorem~3.3]{Na$_2$}. We are interested in the case $p\ne n.$

\medskip
\begin{lemma}\label{lem4}
{\it\, Let $p>n-1,$ $F_1,$ $F_2,$ $F_3$ be three sets in a domain
$D,$ let $D$ contain the spherical ring $A(x_0, a, b),$ $x_0\in
{\Bbb R}^n,$ $0<a< b<\infty,$ let $F_3$ lie in $\overline{B(x_0,
a)},$ and let $\Gamma_{ij}$ be as in Lemma~\ref{lem5}. If one of the
three conditions

\medskip
(1) $F_i$ lies in ${\Bbb R}^n\setminus B(x_0, b),$ $i=1,2;$ (2)
$F_1$ lies in ${\Bbb R}^n\setminus B(x_0, b),$ and $F_2$ is
connected with $d(F_2)> 2b;$ (3) $F_i$ is connected with
$d(F_i)>2b,$ $i=1,2,$ is satisfied, then $$M(\Gamma_{1,
2})>\min\left\{M(\Gamma_{1,3}), M(\Gamma_{2,3}),
c_n\log\frac{b}{a}\right\}$$ whenever $p=n,$ where $c_n$ is a
positive constant depending only on $n.$ If $p\ne n,$
\begin{equation}\label{eq1A}
M_p(\Gamma_{1, 2})>\min\left\{M_p(\Gamma_{1,3}), M_p(\Gamma_{2,3}),
\frac{2^nb_{n,p}}{n-p}(b^{n-p}-a^{n-p})\right\}\,.
\end{equation}
 }
\end{lemma}
\begin{proof}
The case $p=n$ is discussed in detail in~\cite[Theorem~3.3]{Na$_2$}.
We should consider the case $p\ne n.$ We may assume that $F_1,$
$F_2,$ $F_3$ are nonempty sets. If (1) is satisfied, then the
assertion follows directly from Lemma~\ref{lem5} and
Proposition~\ref{pr1}. Assume next that (2) or (3) is satisfied.
Choose $\rho\in {\rm adm\,}\Gamma_{1, 2}.$ If at least one of the
conditions $\int\limits_{\gamma_{1,3}}\rho \,|dx|\geqslant 1/3,$
$\int\limits_{\gamma_{2,3}}\rho\, |dx|\geqslant 1/3$
holds for every rectifiable path $\gamma_{1,3}\in \Gamma_{1,3}$ or
$\gamma_{2,3}\in \Gamma_{2,3},$ respectively, then~(\ref{eq1A})
holds. Otherwise, $\int\limits_{\alpha}\rho \,|dx|\geqslant 1/3$
holds for every rectifiable path $\alpha\in \Gamma(F_1\cup
|\gamma_{1, 3}|, F_2\cup|\gamma_{2, 3}|, D).$ Therefore, since
$S(x_0, t)$ meets both $F_1\cup |\gamma_{1, 3}|$ and $F_2\cup
|\gamma_{2, 3}|$ for $a<t<b$ and since $D$ contains the spherical
ring $A(x_0, a, b),$ we obtain by Proposition~\ref{pr1} that
$M_p(\Gamma_{1,2})\geqslant
\frac{2^nb_{n,p}}{n-p}(b^{n-p}-a^{n-p})\,.
$
Thus, the relation~(\ref{eq1A}) holds, as required. Lemma is
proved.~$\Box$
\end{proof}

\medskip
Finally, we have the following, see \cite[Theorem~1.1]{SKN},
cf.~\cite[Lemma~1.15]{Na$_1$}.

\medskip
\begin{proposition}\label{pr4}
{\it Let $D$ be a domain in ${\Bbb R}^n,$ $n\geqslant 2,$ and let
$n-1<p\leqslant n.$ If $A$ and $A^{\,*}$ are (nondegenerate)
continua in $D,$ then $M_p(\Gamma(A, A^{\,*}, D))>0.$}
\end{proposition}

\section{Main Lemmas}

The following statement generalizes Lemma~\ref{lem3} for the case of
arbitrary continuum $A$ in $D,$ cf.~\cite[Theorem~3.1]{Na$_2$}.

\medskip
\begin{lemma}\label{lem6}
{\it\, Let $n\geqslant 2,$ let $p\in (n-1, n],$ let $x_0\in {\Bbb
R}^n,$ let $\varepsilon_1>0,$ and let $A^{\,*}$ be a
(non-degenerate) continuum in $D\subset {\Bbb R}^n.$ Let $r>0$ and
let $C_j,$ $j=1,2,\ldots, $ be a sequence of continua in $B(x_0,
\varepsilon_1)$ such that $h(C_j)\geqslant r,$
$h(C_j)=\sup\limits_{x, y\in C_j}h(x, y).$ Then there is
$R^{\,*}_0>0$ such that $M_p(\Gamma(C_j, A^{\,*}, D))\geqslant
R^{\,*}_0\quad\forall\,\,j\in {\Bbb N}\,.$}
\end{lemma}

\begin{proof}
Due to Lemma~\ref{lem3}, for any continuum $A\subset B(x_0,
\varepsilon_1)$ there exists $R_0>0$ such that
$
M_p(\Gamma(C_j, A, D))\geqslant R_0\quad\forall\,\,j\in {\Bbb N}\,.
$
Thus, we may consider that $A^{\,*}\subset {\Bbb R}^n\setminus
B(x_0, \varepsilon_1).$ In particular, $A\cap A^{\,*}= \varnothing.$
Choose $r_0>0$ such that $0<4r_0<\min\{r, d(A, A^{\,*})\}.$ Let
$A_1,\ldots, A_q$ be a finite covering of $A$ by closed balls
centered at the points $a_i\in A_i$ and of the radius $r_0,$
$i=1,\ldots, q.$ Denote $\delta_i:=M_p(\Gamma(A_i, A^{\,*}, D)).$
Recall that, $p$-modulus of families of paths joining two continua
in $D$ is positive for $n-1<p\leqslant n$ (see
Proposition~\ref{pr4}). Now, we set
$R_0^{\,*}=3^{\,-n}\min\left\{R_0/q, \delta_1,\ldots, \delta_q,
c_n\log 2\right\}$
for $p=n,$ where $c_n$ is a constant from Lemma~\ref{lem4}, and
$$R_0^{\,*}= 3^{\,-p}\min\left\{R_0/q, \delta_1,\ldots,
\delta_q, \frac{2^nb_{n,p}}{n-p}((2r_0)^{n-p}-r_0^{n-p})\right\}$$
for $p\ne n,$ where $b_{n,p}$ is a constant from
Proposition~\ref{pr1}. Fix $j\in {\Bbb N}.$ Due to the subadditivity
of the modulus of families of paths,
\begin{equation}\label{eq26***}
0<R_0 \leqslant M_p(\Gamma(A_i, C_j, D))\leqslant
M_p\left(\Gamma\left(\bigcup\limits_{i=1}^q A_i, C_j,
D\right)\right)\leqslant \sum\limits_{i=1}^q M_p(\Gamma(A_i, C_j,
D))\,.
\end{equation}
It follows from~(\ref{eq26***}) that $M_p(\Gamma(A_{i_0}, C_j,
D))\geqslant \delta/q$ at least for some $i_0\in \{1,\ldots,q\}.$
Since $A^{*}\cap B(a_i, 2r_0)=\varnothing$ and since $d(C_j)>4r_0,$
the assertion follows from Lemma~\ref{lem4} setting $F_1=A^{\,*},$
$F_2=C_j$ and $F_3=A_i.$~$\Box$
\end{proof}

Given a Lebesgue measurable function $Q:{\Bbb R}^n\rightarrow [0,
\infty]$ and a point $x_0\in {\Bbb R}^n$ we set
$
q_{x_0}(t)=\frac{1}{\omega_{n-1}t^{n-1}} \int\limits_{S(x_0,
t)}Q(x)\,d\mathcal{H}^{n-1}\,,
$
where $\mathcal{H}^{n-1}$ denotes $(n-1)$-dimensional Hausdorff
measure. The following lemma is proved in \cite[Lemma~2.1]{Sev$_2$}.

\medskip
\begin{lemma}\label{lem1}
{\it\, Let $1\leqslant p\leqslant n,$ and let $\Phi:[0,
\infty]\rightarrow [0, \infty] $ be a strictly increasing convex
function such that the relation $\int\limits_{\delta_0}^{\infty}
\frac{d\tau}{\tau\left[\Phi^{-1}(\tau)\right]^{\frac{1}{p-1}}}=
\infty$ holds for some $\delta_0>\tau_0:=\Phi(0).$ Let $\frak{Q}$ be
a family of functions $Q:{\Bbb R}^n\rightarrow [0, \infty]$ such
that
$
\int\limits_D\Phi(Q(x))\frac{dm(x)}{\left(1+|x|^2\right)^n}\
\leqslant M_0<\infty
$
for some $0<M_0<\infty.$ Now, for any $0<r_0<1$ and for every
$\sigma>0$ there exists $0<r_*=r_*(\sigma, r_0, \Phi)<r_0$ such that
$\int\limits_{\varepsilon}^{r_0}\frac{dt}{t^{\frac{n-1}{p-1}}q^{\frac{1}{p-1}}_{x_0}(t)}\geqslant
\sigma,$ $\varepsilon\in (0, r_*),$
for any $Q\in \frak{Q}.$ }
\end{lemma}

\medskip
Let $D\subset {\Bbb R}^n,$ $f:D\rightarrow {\Bbb R}^n$ be a discrete
open mapping, $\beta: [a,\,b)\rightarrow {\Bbb R}^n$ be a path, and
$x\in\,f^{\,-1}(\beta(a)).$ A path $\alpha: [a,\,c)\rightarrow D$ is
called a {\it maximal $f$-lifting} of $\beta$ starting at $x,$ if
$(1)\quad \alpha(a)=x\,;$ $(2)\quad f\circ\alpha=\beta|_{[a,\,c)};$
$(3)$\quad for $c<c^{\prime}\leqslant b,$ there is no a path
$\alpha^{\prime}: [a,\,c^{\prime})\rightarrow D$ such that
$\alpha=\alpha^{\prime}|_{[a,\,c)}$ and $f\circ
\alpha^{\,\prime}=\beta|_{[a,\,c^{\prime})}.$ If $\beta:[a,
b)\rightarrow\overline{{\Bbb R}^n}$ is a path and if
$C\subset\overline{{\Bbb R}^n},$ we say that $\beta\rightarrow C$ as
$t\rightarrow b,$ if the spherical distance $h(\beta(t),
C)\rightarrow 0$ as $t\rightarrow b$ (see \cite[Section~3.11]{MRV}),
where $h(\beta(t), C)=\inf\limits_{x\in C}h(\beta(t), x).$ The
following assertion holds (see~\cite[Lemma~3.12]{MRV}).

\medskip
\begin{proposition}\label{pr3}
{\it Let $f:D\rightarrow {\Bbb R}^n,$ $n\geqslant 2,$ be an open
discrete mapping, let $x_0\in D,$ and let $\beta: [a,\,b)\rightarrow
{\Bbb R}^n$ be a path such that $\beta(a)=f(x_0)$ and such that
either $\lim\limits_{t\rightarrow b}\beta(t)$ exists, or
$\beta(t)\rightarrow \partial f(D)$ as $t\rightarrow b.$ Then
$\beta$ has a maximal $f$-lifting $\alpha: [a,\,c)\rightarrow D$
starting at $x_0.$ If $\alpha(t)\rightarrow x_1\in D$ as
$t\rightarrow c,$ then $c=b$ and $f(x_1)=\lim\limits_{t\rightarrow
b}\beta(t).$ Otherwise $\alpha(t)\rightarrow \partial D$ as
$t\rightarrow c.$}
\end{proposition}

\medskip
{\it Proof of Theorem~\ref{th1}}. Let us prove the theorem by
contradiction. Assume that its conclusion is wrong, i.e., there is a
compactum $K$ in $D$ for which there exists $0<\varepsilon_1<{\rm
dist}\,(K,
\partial D)$ such that for any $m\in {\Bbb N}$ there exists
$f_m\in \frak{F}^{Q, p}_{a, b, \delta}(D)$ and $x_m\in K$ such that
$B_h\left(f_m(x_m), \frac{1}{m}\right)\setminus f_m(B(x_m,
\varepsilon_1))\ne \varnothing,$
where $B_h\left(f_m(x_m), \frac{1}{m}\right)=\{w\in \overline{{\Bbb
R}^n}: h(w, f_m(x_m))<\frac{1}{m}\}.$ Let $y_m\in B_h(f_m(x_m),
\frac{1}{m})\setminus f_m(B(x_m, \varepsilon_0)).$ Due to the
compactness of $\overline{{\Bbb R}^n}$ we may consider that
$y_m\rightarrow y_0$ as $m\rightarrow\infty,$ where $y_0\in
\overline{{\Bbb R}^n}.$ Then also $f_m(x_m)\rightarrow y_0$ as
$m\rightarrow\infty.$ Let us firstly consider that $y_0\ne\infty.$
Passing to a subsequence, if necessary, we may consider that
$|f_m(x_m)-y_m|<1/m,$ $m=1,2,\ldots .$

\medskip
Let us join the points $a$ and $b$ with a path $\gamma:[0,
1]\rightarrow D,$ $\gamma(0)=a,$ $\gamma(1)=b,$ in $D.$ It follows
from the conditions of the lemma that $h(f_m(\gamma))\geqslant
\delta$ for any $m=1,2,\ldots ,$ $f_m\in \frak{F}^{Q, p}_{a, b,
\delta}(D).$ We may consider that the sequences $f_m(a)$ and
$f_m(b)$ converge to some points $z_1$ and $z_2$ as $m\rightarrow
\infty$ because $\overline{{\Bbb R}^n}$ is a compact space. Due to
the condition $h(f_m(a), f_m(b))\geqslant\delta,$ at least one of
the above points does not coincide with $y_0.$ Without loss of
generality, we may consider that $z_1\ne y_0.$ Since $f_m$ is
equicontinuous at $a,$ given $\sigma>0$ there is $\chi=\chi(\sigma)$
such that $h(f_m(x), f_m(a))<\sigma$ for $|x-a|<\chi.$ We may chose
numbers $r_1, r_2>0$ such small that
\begin{equation}\label{eq6A}
B_h(z_1, r_1)\cap B(y_0, r_2)=\varnothing\,.
\end{equation}
By the triangle inequality
$h(f_m(x), z_1)\leqslant h(f_m(x), f_m(a))+ h(f_m(a), z_1)<\sigma+
h(f_m(a), z_1)$
for $|x-a|<\chi.$ Since $h(f_m(a), z_1)\rightarrow 0$ as
$m\rightarrow\infty$ the latter relation implies that $f_m(x)\in
B_h(z_1, r_1)$ for sufficiently large $m$ and $\sigma=r_1/2.$ Let
$E=\{|x-a|<\chi\},$ where $\chi$ is mentioned above.

Join the points $y_m$ and $f_m(x_0)$ by the segment $\beta_m:[0,
1]\rightarrow \overline{B(f_m(x_m), 1/m)}$ such that
$\beta_m(0)=f_m(x_m)$ and $\beta_m(0)=y_m.$ Let $\alpha_m,$
$\alpha_m:[0, c_m)\rightarrow B(x_m, \varepsilon_1),$ be a maximal
$f_m$-lifting of $\beta_m$ in $B(x_m, \varepsilon_1)$ starting at
$x_0,$ which exists by Proposition~\ref{pr3}. By the same
Proposition either $\alpha_m(t)\rightarrow x_1\in B(x_m,
\varepsilon_1)$ as $t\rightarrow c_m-0$ (in this case, $c_m=1$ and
$f_m(x_1)=y_m$), or $\alpha_m(t)\rightarrow S(x_m, \varepsilon_1)$
as $t\rightarrow c_m.$ Observe that, the first situation is
excluded. Indeed, if $f_m(x_1)=y_m,$ then $y_m\in f_m(B(x_m,
\varepsilon_1)),$ that contradicts the choice of $y_m.$ Thus,
$\alpha_m(t)\rightarrow S(x_m, \varepsilon_1)$ as $t\rightarrow
c_m.$ Observe that, $\overline{|\alpha_m|}$ is a continuum in
$\overline{B(x_m, \varepsilon_1)}$ and
$h(\overline{|\alpha_m|})\geqslant h(x_m, S(x_m, \varepsilon_1)).$

Without loss of generality we may consider that $x_m\rightarrow
x_0\in K,$ $m\rightarrow\infty.$ Let us show that
$\overline{|\alpha_m|}\in B(x_0, \varepsilon^{\,*}_1),$ where
$\varepsilon^{\,*}_1>0$ is some number with
$0<\varepsilon^{\,*}_1<{\rm dist}\,(K,
\partial D).$ Indeed, let $x\in \overline{|\alpha_m|}.$ Since
$\overline{|\alpha_m|}$ is a continuum in $\overline{B(x_m,
\varepsilon_1)},$ by the triangle inequality we obtain that
$|x-x_0|\leqslant |x-x_m|+|x_m-x_0|<\varepsilon_1+|x_m-x_0|.$
Since $x_m-x_0\rightarrow 0$ as $m\rightarrow\infty$ and since
$0<\varepsilon_1<{\rm dist}\,(K,
\partial D),$ we may chose $r_*>0$ such small that $\varepsilon^{\,*}_1:=
\varepsilon_1+r_*<{\rm dist}\,(K,
\partial D).$ Thus, $\overline{|\alpha_m|}\in B(x_0, \varepsilon^{\,*}_1),$
$0<\varepsilon^{\,*}_1<{\rm dist}\,(K,
\partial D),$ as required.

Recall that, $h(\overline{|\alpha_m|})\geqslant h(x_m, S(x_m,
\varepsilon_1)).$ Let $h(x_m, S(x_m, \varepsilon_1))=h(x_m, w_m),$
where $w_m\in S(x_m, \varepsilon_1).$ Now, by the definition of a
chordal metrics,
\begin{equation}\label{eq3}
h(w_m, x_0)=\frac{|w_m-x_0|}{\sqrt{1+|w_m|^2}\sqrt{1+|x_0|^2}}
\end{equation}
and since by the triangle inequality $|w_m|\leqslant
|w_m-x_0|+|x_0|\leqslant |w_m-x_m|+|x_m-x_0|+|x_0|\leqslant
2\varepsilon_1+|x_0|$ for sufficiently large $m,$ we obtain
from~(\ref{eq3}) that
\begin{equation}\label{eq4_A}
h(x_m, w_m)=
\frac{|w_m-x_m|}{\sqrt{1+|w_m|^2}\sqrt{1+|x_0|^2}}\geqslant
\frac{\varepsilon_1}{\sqrt{1+|2\varepsilon_1+|x_0||^2}\sqrt{1+|x_0|^2}}:=r\,.
\end{equation}
Thus, $h(\overline{|\alpha_m|})\geqslant r$ for sufficiently large
$m\in {\Bbb N}.$ Now, we apply Lemma~\ref{lem6} for
$A:=E=\{|x-x_0|<\chi\},$ $C_m:=|\alpha_m|$ and $r$ which is defined
in~(\ref{eq4_A}). By this lemma we may find $R_0>0$ such that
\begin{equation}\label{eq5}
M_p(\Gamma(\overline{|\alpha_m|}, E, D))\geqslant R_0\,,\qquad
m=1,2,\ldots\,.
\end{equation}

Let us show that the relation~(\ref{eq5}) contradicts the definition
of the mapping $f_m$ in~(\ref{eq2*A})--(\ref{eqA2}). Indeed, since
$f_m(x_m)\rightarrow y_0$ as $m\rightarrow\infty,$ for any $k\in
{\Bbb N}$ there is a number $m_k\in {\Bbb N}$ such that
\begin{equation}\label{eq6}
B(f_{m_k}(x_{m_k}), 1/m_k)\subset B(y_0, 2^{\,-k})\,.
\end{equation}
Since $|\beta_m|\in B(f_m(x_m), 1/m),$ by~(\ref{eq6}) we obtain that
\begin{equation}\label{eq7}
|\beta_{m_k}|\subset B(y_0, 2^{\,-k})\,,\qquad k=1,2,\ldots\,.
\end{equation}
Let $k_0\in {\Bbb N}$ be such that $2^{\,-k}<\varepsilon_2,$ where
$\varepsilon_2<r_2$ and $r_2$ corresponds to~(\ref{eq6A}), and let
$\Gamma_k:=\Gamma(|\alpha_{m_k}|, E, D).$ In this case, $\Gamma_k>
\Gamma_{f_{m_k}}(y_0, 2^{\,-k}, \varepsilon_2)$ for sufficiently
large $k\in {\Bbb N},$ see~(\ref{eq6A}) and~(\ref{eq7}) (cf.~
Figure~\ref{fig1}).
\begin{figure}[h]
\centerline{\includegraphics[scale=0.4]{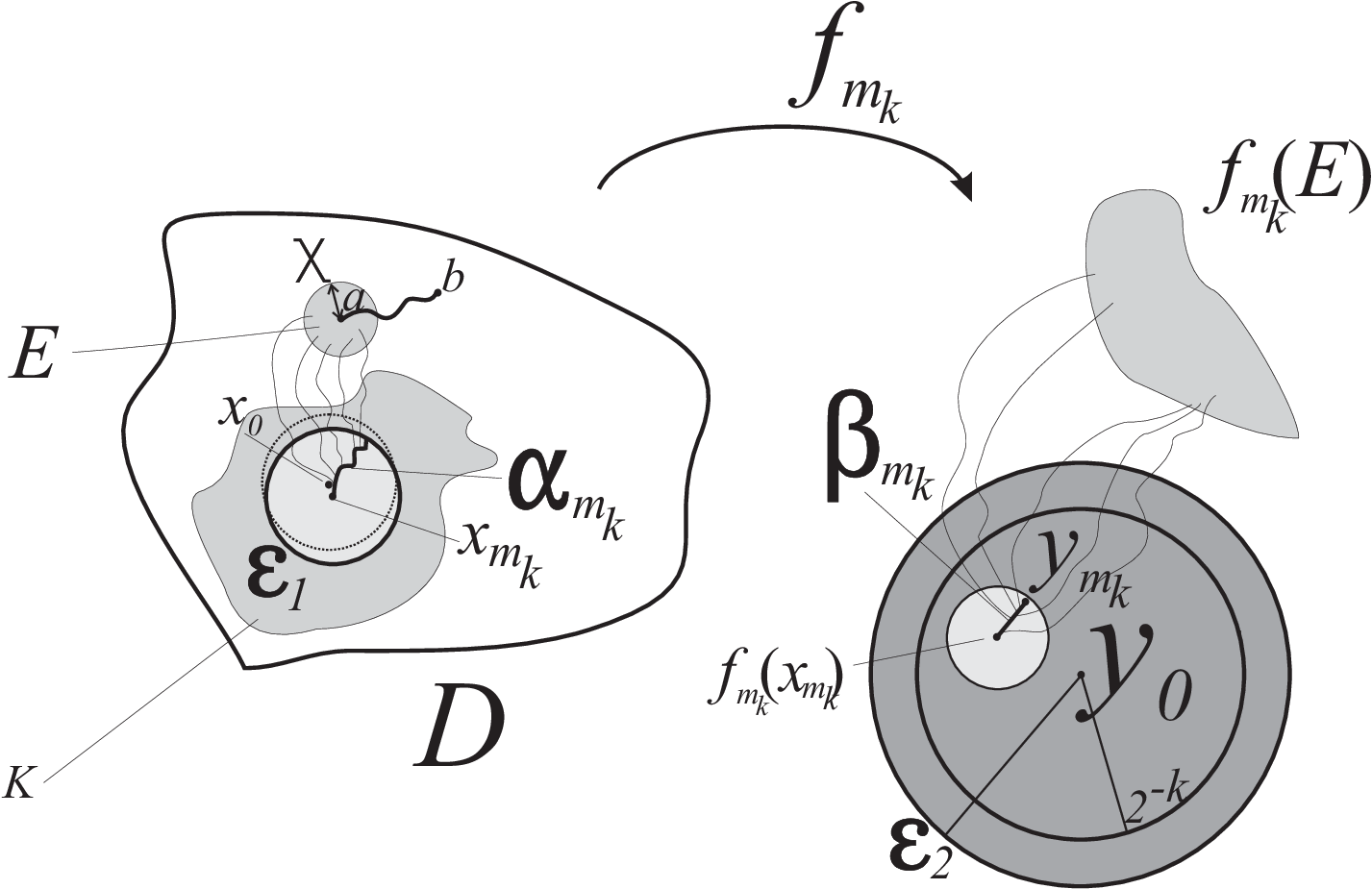}}
\caption{To the proof of Lemma~\ref{lem1}}\label{fig1}
\end{figure}
Now, by the definition of $f_{m_k}$ in~(\ref{eq2*A})--(\ref{eqA2}),
we obtain that
\begin{gather} M_p(\Gamma_k)=M_p(\Gamma(|\alpha_{m_k}|, E, D) )\leqslant
M_p(\Gamma_{f_{m_k}}(y_0, 2^{\,-k}, \varepsilon_2))\nonumber\\
\label{eq3J} \leqslant \int\limits_{A(y_0, 2^{\,-k}, \varepsilon_2)}
Q_{f_{m_k}}(y)\cdot\eta^{\,p}(|y-y_0|)\,dm(y)
\leqslant \int\limits_{A(y_0, 2^{\,-k}, \varepsilon_2)}\nonumber
\widetilde{Q}_{f_{m_k}}(y)\cdot\eta^{\,p}(|y-y_0|)\,dm(y)\,,
\end{gather}
where
$$\widetilde{Q}_{f_{m_k}}(y)=\begin{cases}Q_{f_{m_k}}(y), & Q_{f_{m_k}}(y)\geqslant 1\\
 1, & Q_{f_{m_k}}(y)<1\end{cases}\,,$$
and $\eta:(2^{\,-k}, \varepsilon_2)\rightarrow {\Bbb R}$ is an
arbitrary nonnegative Lebesgue measurable function with
$\int\limits_{2^{\,-k}}^{\varepsilon_2}\eta(r)\,dr\geqslant 1.$
Now, we set
$I_k=I(y_0, 2^{\,-k},
\varepsilon_2)=\int\limits_{2^{\,-k}}^{\varepsilon_2}\
\frac{dr}{r^{\frac{n-1}{p-1}}\widetilde{q}_{k,
y_0}^{\frac{1}{p-1}}(r)},$
where
$$\widetilde{q}_{k, y_0}(r)=\frac{1}{\omega_{n-1}r^{n-1}} \int\limits_{S(y_0,
r)}\widetilde{Q}_{f_{m_k}}(y)\,d\mathcal{H}^{n-1}\,.$$
By Lemma~\ref{lem1} and Remark~\ref{rem1}
\begin{equation}\label{eq2_1}
I_k=\int\limits_{2^{\,-k}}^{\varepsilon_2}\
\frac{dr}{r^{\frac{n-1}{p-1}}\widetilde{q}_{k,
y_0}^{\frac{1}{p-1}}(r)}\rightarrow\infty
\end{equation}
as $k\rightarrow\infty.$
Since $\widetilde{q}_{k, y_0}(r)\geqslant 1$ for a.e. $r$ and
by~(\ref{eq2_1}), $0<I_k<\infty$ for sufficiently large $k\in {\Bbb
N}.$ Set
$$ \psi_k(t)= \left \{\begin{array}{rr}
1/[t^{\frac{n-1}{p-1}}\widetilde{q}_{k, y_0}^{\frac{1}{p-1}}(t)], &
t\in (2^{\,-k}, \varepsilon_2)\ ,
\\ 0,  &  t\notin (2^{\,-k}, \varepsilon_2)\ .
\end{array} \right. $$
Let $\eta_k(t)=\psi_k(t)/I_k$ for $t\in (2^{\,-k}, \varepsilon_2)$
and $\eta_k(t)=0$ otherwise. Now, $\eta_k$ satisfies~(\ref{eqA2})
for $r_1=2^{\,-k}$ and $r_2=\varepsilon_2.$ Therefore,
by~(\ref{eq3J}), (\ref{eq2_1}) and by Fubini's theorem,
\begin{equation}\label{eq3_1}
M_p(\Gamma_k)\leqslant\frac{1}{I^{p}_k}\int\limits_{A(y_0, 2^{\,-k},
\varepsilon_2)}
Q(y)\cdot\psi_k^p(|y-y_0|)\,dm(y)=\frac{\omega_{n-1}}{I^{p-1}_k}\rightarrow
0
\end{equation}
as $k\rightarrow\infty.$ The relation~(\ref{eq3_1}) contradicts
with~(\ref{eq5}). The contradiction obtained above proves the lemma.

\medskip
Let now $y_0=\infty.$ Now, we set $\widetilde{f}_m:=\psi\circ f_m$
and $\widetilde{y}_m:=\psi(y_m),$ where $\psi(y)=\frac{y}{|y|^2}.$
Since $y_0=\infty,$ we obtain that $\widetilde{y}_m\rightarrow 0$ as
$m\rightarrow\infty.$ Besides that, since $f_m(x_m)-y_m\rightarrow
0$ as $m\rightarrow\infty,$ by the triangle inequality,
$ |\widetilde{f}_m(x_m)-\widetilde{y}_m|=|\psi(f_m(x_m))-\psi(y_m)|
\leqslant |\psi(f_m(x_m))|+|\psi(y_m)|\rightarrow 0,$ $m
\rightarrow\infty.$
Therefore, $\widetilde{f}_m(x_m)-\widetilde{y}_m\rightarrow 0$ as
$m\rightarrow\infty.$ Passing to a subsequence, we may also consider
that $|\widetilde{f}_m(x_m)-\widetilde{y}_m|<\frac{1}{m},$
$m=1,2,\ldots.$ In addition, by the assumption of the lemma, the
mappings $\widetilde{f}_m$ satisfy the
relations~(\ref{eq2*A})--(\ref{eqA2}) at the origin with a new
function $\widetilde{Q}(y):=Q\left(\frac{y}{|y_2|}\right).$
The conformal change of the variables $z=\frac{y}{|y|^2}$
corresponds to the jacobian $J(z, y)=\frac{1}{|y|^{2n}}.$ So,
observe that,
\begin{equation}\label{eq3A} \int\limits_{{\Bbb
R}^n}\Phi\left(Q_{f_m}\left(\frac{y}{|y|^2}\right)\right)\cdot\frac{dm(y)}{(1+|y|^2)^n}=
\int\limits_{{\Bbb
R}^n}\Phi(Q_{f_m}(z))\cdot\frac{dm(z)}{(1+|z|^2)^n}<\infty\,.
\end{equation}
The relation~(\ref{eq3A}) implies the possibility of applying above
technique to $Q_{f_m}\left(\frac{y}{|y|^2}\right)$ instead of
$Q_{f_m}(y).$ In particular, the relation~(\ref{eq3A}) together with
(\ref{eq2A}) implies by Lemma~\ref{lem1} that (\ref{eq2_1}) holds
for $\widetilde{\widetilde{q}}_{k,
y_0}(r):=\frac{1}{\omega_{n-1}r^{n-1}} \int\limits_{S(y_0,
r)}\widetilde{Q}_{f_{m_k}}\left(\frac{y}{|y|^2}\right)\,d\mathcal{H}^{n-1}$
instead of $\widetilde{q}_{k, y_0}.$ Finally, since the family
$\{f_m\}_{m=1}^{\infty}$ is equicontinuous at $a$ and $b,$ the
family $\{\widetilde{f}_m\}_{m=1}^{\infty}$ is also equicontinuous
at these points because $\widetilde{f}_m=\psi\circ f_m$ and $\psi$
is a fixed continuous function.

Now, taking into account what has been said, we repeat verbatim the
proof of Theorem~\ref{th1} for the case $y_0=0$ and for
$\widetilde{f}_m$ instead $f_m,$ $m=1,2,\ldots .$  Repeating this
proof, we obtain the desired conclusion.~$\Box$

\medskip
The following result is from~\cite[Proposition~4.4.I]{Ri}.

\medskip
\begin{proposition}\label{pr6}
{\it\,Let $f$ and $g$ be homotopic via a homotopy $h_t,$ $t\in
[0,1],$ $h_0=f$, $h_1=g.$ Suppose further that $y$ is $(h_t,
U)$-admissible for all $t\in [0,1].$ Then $\mu(y, f, U)=\mu(y, g,
U)$ (where $\mu(y, f, G)$ denotes the topological degree of $f$ at
$y$ with respect to $G$).}
\end{proposition}

\medskip
Given a mapping $f:D\,\rightarrow\,{\Bbb R}^n,$ a set $E\subset D$
and $y\,\in\,{\Bbb R}^n,$ we define the {\it multiplicity function
$N(y,f,E)$} as a number of preimages of the point $y$ in a set $E,$
i.e. $N(y,f,E)\,=\,{\rm card}\,\left\{x\in E: f(x)=y\right\}\,,$
\begin{equation}\label{eq1G}
N(f,E)\,=\,\sup\limits_{y\in{\Bbb R}^n}\,N(y,f,E)\,.
\end{equation}
The following result is from \cite[Proposition~4.10.I]{Ri}.

\medskip
\begin{proposition}\label{pr7}
{\it\, The relation $N(y, f, G)\leqslant\mu(y, f, G)$ is true for
$y\in {\Bbb R}^n\setminus f(\partial G)$ and any compact subdomain
$G$ of $D.$ }
\end{proposition}

\medskip
{\it Proof of Theorem~\ref{th2}.} Assume that $f$ is not a constant.
Let us to show that $f$ is open. Let $A$ be an open set and let
$x_0\in A.$ We need to show that, there is $\varepsilon^*>0$ such
that $B(f(x_0), \varepsilon^*)\subset f(A).$ Since $A$ is open,
there is $\varepsilon_1>0$ such that $\overline{B(x_0,
\varepsilon_1)}\subset A.$ Since $f$ is not constant, there are $a,
b\in D$ such that $f(a)\ne f(b).$ Let us join the points $a$ and $b$
by a path $\gamma$ in $D.$ We set $E:=|\gamma|.$ Now, $h(f_m(a),
f_m(b))\geqslant \frac{1}{2}\cdot h(f(a), f(b)):=\delta$ for
sufficiently large $m\in {\Bbb N}.$ By Theorem~\ref{th1} there is
$r_0>0,$ which does not depend on $m,$ such that
$B_h(f_m(x_0), r_0)\subset f_m(B(x_0, \varepsilon_1)),$
$m=1,2,\ldots.$

Set $\varepsilon^*:=r_0/2.$ Let $y\in B_h(f(x_0), r_0/2).$ Since by
the assumption $f_m(x)\rightarrow f(x)$ locally uniformly in $D,$ by
the triangle inequality we obtain that
$h(f_m(x_0), y)\leqslant h(f_m(x_0), f(x_0))+h(f(x_0), y)<r_0$
for sufficiently large $m\in {\Bbb N}.$ Thus, $y\in B_h(f_m(x_0),
r_0)\subset f_m(B(x_0, \varepsilon_1)).$ Consequently, $y=f_m(x_m)$
for some $x_m\in B(x_0, \varepsilon_1).$ Due to the compactness of
$\overline{B(x_0, \varepsilon_1)},$ we may consider that
$x_m\rightarrow z_0\in \overline{B(x_0, \varepsilon_1)}$ as
$m\rightarrow\infty.$ By the continuity of $f$ in $A,$ since
$\overline{B(x_0, \varepsilon_1)}\subset A,$ we obtain that
$f(x_m)\rightarrow f(z_0)$ as $m\rightarrow\infty.$ So, we have that
$f(x_m)\rightarrow f(z_0)$ as $m\rightarrow\infty$ and
simultaneously $y=f_m(x_m)$ for sufficiently large $m\in {\Bbb N}.$
Thus,
$|y-f(z_0)|=|f_m(x_m)-f(z_0)|\leqslant
|f_m(x_m)-f(x_m)|+|f(x_m)-f(z_0)|\rightarrow 0$ as
$m\rightarrow\infty.$ Thus, $y=f(z_0)\in f(\overline{B(x_0,
\varepsilon_1)})\subset f(A).$ So, $y\in f(A),$ i.e., $B(f(x_0),
r_0/2)\subset f(A),$ as required.

\medskip
We will prove the discreteness of the mapping using methodology
from~\cite{Cr}. Assume the contrary. Then there is $x_0 \in D$ and a
sequence $x_k\in D,$ $k=1,2,\ldots ,$ $x_k \ne x_0,$ such that
$x_k\rightarrow x_0$ as $k\rightarrow \infty$ and $f(x_k)=f(x_0).$
Observe that $E_0=\{x\in D: f(x)=f(x_0)\}$ is closed in $D$ by the
continuity of $f$ and does not coincide with $D,$ because $f\not
\equiv const.$ Thus, we may assume that $x_0$ may be replaced by non
isolated boundary point of $E_0.$ Let $\varepsilon_0>0$ be such that
$\overline{B(x_0, \varepsilon_0)}\subset D.$ By the assumption, $f$
is not identically a constant in $B(x_0, \varepsilon)$ for any
$0<\varepsilon<\varepsilon_0.$

\medskip
We show that for every $0<\varepsilon<\varepsilon_0$ there exist
$0<t<\varepsilon_0$ such that $f(x_0)\not\in f(S(x_0, t)).$
Otherwise, there is $0<\varepsilon<\varepsilon_0$  and $x_t\in
S(x_0, t)$ with $f(x_t)=f(x_0)$ for every $0<t<\varepsilon.$ Since
$f$ is not a constant in $B(x_0, \varepsilon),$ we may find $z_2\in
B(x_0, \varepsilon)$ and $\varepsilon_2>0$ such that
$\overline{B(z_2, \varepsilon_2)}\subset B(x_0, \varepsilon)$ and
$\overline{B(f(x_0), \varepsilon_2)}\cap f(B(z_2,
\varepsilon_2))=\varnothing.$ Since $f_m$ converges to $f$ locally
uniformly, we also have that
\begin{equation}\label{eq1C}
d(f_m(B(z_2, \varepsilon_2)), f(x_0))\geqslant \varepsilon_1>0
\end{equation}
for sufficiently large $m\in {\Bbb N}$ and some $\varepsilon_1>0.$
Let $E =\{x_t\}_{t\in (0, \varepsilon)}\,,$
$\Gamma=\Gamma(\overline{B(z_2, \varepsilon_2)}, E, B(x_0,
\varepsilon))\,.$ Since $\overline{B(z_2, \varepsilon_2)}$ is a
continuum, there are $0<a<b<\varepsilon$ such that $E\cap S(x_0,
t)\ne\varnothing\ne \overline{B(z_2, \varepsilon_2)}\cap S(x_0, t)$
for every $a<t<b.$ By Proposition~\ref{pr1}
\begin{equation}\label{eq1D}
M_p(\Gamma)\geqslant\alpha>0
\end{equation}
for some $\alpha>0.$ On the other hand, let $F:=f^{\,-1}(f(x_0))\cap
\overline{B(x_0, \varepsilon)}.$ Now, $F$ is a compact set in $D.$
We set $r_m=\max\{0, \widetilde{r}_m\},$ where
$\widetilde{r}_m=\sup\limits_{x\in f_m(E)}|x-f(x_0)|.$ By the
triangle inequality,
$|f_m(x)-f_m(x_0)|\leqslant |f_m(x)-f(x)|+|f(x)-f(x_0)|\rightarrow
0$ as $m\rightarrow\infty$ uniformly over $x\in F.$ Thus,
$\widetilde{r}_m=\sup\limits_{x\in f_m(E)}|x-f(x_0)|\leqslant
\sup\limits_{x\in f_m(F)}|x-f(x_0)|=|f_m(z_m)-f(x_0)|\rightarrow 0
$
as $m\rightarrow\infty,$ where $z_m\in F$ is such that
$\sup\limits_{x\in f_m(F)}|x-f(x_0)|=|f_m(z_m)-f(x_0)|.$ Therefore,
$f_m(E)\subset B(f(x_0), r_m),$ where $r_m\rightarrow 0$ as
$m\rightarrow\infty.$ Now, by~(\ref{eq1C}),
$f_m(\Gamma)>\Gamma(S(f(x_0), r_m), S(f(x_0), \varepsilon_1),
A(f(x_0), r_m, \varepsilon_1))$
for sufficiently large $m\in {\Bbb N}.$ The latter implies that
\begin{equation}\label{eq1E}
\Gamma>\Gamma_{f_m}(f(x_0), r_m, \varepsilon_1)\,.
\end{equation}
The relation~(\ref{eq1E}) together with the definition
of $f_{m_k}$ in~(\ref{eq2*A})--(\ref{eqA2}) gives that
\begin{equation}\label{eq3D}
M_p(\Gamma)\leqslant M_p(\Gamma_{f_m}(f(x_0), r_m,
\varepsilon_1))\leqslant \int\limits_{A(f(x_0), r_m, \varepsilon_1)}
\widetilde{Q}_{f_{m}}(y)\cdot\eta^{\,p}(|y-f(x_0)|)\,dm(y)\,,
\end{equation}
where
$$\widetilde{Q}_{f_{m}}(y)=\begin{cases}Q_{f_{m}}(y), & Q_{f_{m}}(y)\geqslant 1\\
 1, & Q_{f_{m}}(y)<1\end{cases}\,,$$
whenever $\eta:(r_m, \varepsilon_1)\rightarrow {\Bbb R}$ is an
arbitrary nonnegative Lebesgue measurable function with
$\int\limits_{r_m}^{\varepsilon_2}\eta(r)\,dr\geqslant 1.$
Now, we set $I_m=I(f(x_0), r_m,
\varepsilon_1)=\int\limits_{r_m}^{\varepsilon_1}\
\frac{dr}{r^{\frac{n-1}{p-1}}\widetilde{q}_{m,
f(x_0)}^{\frac{1}{p-1}}(r)}\,,$
where
$\widetilde{q}_{m, f(x_0)}(r)=\frac{1}{\omega_{n-1}r^{n-1}}
\int\limits_{S(f(x_0),
r)}\widetilde{Q}_{f_{m}}(y)\,d\mathcal{H}^{n-1}\,.$
By Lemma~\ref{lem1} and Remark~\ref{rem1}
\begin{equation}\label{eq2_2}
I_m=\int\limits_{r_m}^{\varepsilon_1}\
\frac{dr}{r^{\frac{n-1}{p-1}}\widetilde{q}_{m,
f(x_0)}^{\frac{1}{p-1}}(r)}\rightarrow\infty
\end{equation}
as $m\rightarrow\infty.$
Since $\widetilde{q}_{m, f(x_0)}(r)\geqslant 1$ for a.e. $r$ and
by~(\ref{eq2_2}), $0<I_m<\infty$ for sufficiently large $m\in {\Bbb
N}.$ Set
$$ \psi_m(t)= \left \{\begin{array}{rr}
1/[t^{\frac{n-1}{p-1}}\widetilde{q}_{m, f(x_0)}^{\frac{1}{p-1}}(t)],
& t\in (r_m, \varepsilon_1)\ ,
\\ 0,  &  t\notin (r_m, \varepsilon_1)\ .
\end{array} \right. $$
Let $\eta_m(t)=\psi_m(t)/I_m$ for $t\in (r_m, \varepsilon_1)$ and
$\eta_m(t)=0$ otherwise. Therefore, $\eta_m$ satisfies~(\ref{eqA2})
for $r_1=r_m$ and $r_2=\varepsilon_1.$ Now, by~(\ref{eq3D}),
(\ref{eq2_2}) and by Fubini's theorem,
\begin{equation}\label{eq3E}
M_p(\Gamma)\leqslant\frac{1}{I^{p}_m}\int\limits_{A(f(x_0), r_m,
\varepsilon_1)}
Q(y)\cdot\psi_m^p(|y-f(x_0)|)\,dm(y)=\frac{\omega_{n-1}}{I^{p-1}_m}\rightarrow
0
\end{equation}
as $m\rightarrow\infty.$ The relation~(\ref{eq3E}) contradicts
with~(\ref{eq1D}). Thus, for every $0<\varepsilon<\varepsilon_0$
there exist $0<t<\varepsilon$ such that $f(x_0)\not\in f(S(x_0,
t)),$ as required.

\medskip
Let $0<\rho<\varepsilon_0,$ let $0<\varepsilon<\rho$ be such that
$y=f(x_0)\not\in f(S(x_0, \varepsilon))$ and let $\beta=d(y,
f(S(x_0, \varepsilon)))>0.$ By Proposition~\ref{pr6},
$n_m:=\mu(y, f_m, B(x_0, \varepsilon))=\mu(y, f, B(x_0,
\varepsilon))=q(y)$
for sufficiently large $m\in {\Bbb N}.$ Now, by
Proposition~\ref{pr7}
\begin{equation}\label{eq7_A}
N(y, f_m, B(x_0, \varepsilon))\leqslant n_m=q(y)<\infty
\end{equation}
for sufficiently large $m=1,2,\ldots .$ Since by the assumption $f$
is not discrete, we may find $x_1,\ldots, x_j\in f^{\,-1}f(x_0)$ in
$B(x_0, \varepsilon)$ with $j>q=q(y).$ We may find $\varepsilon_k>0$
such that $\overline{B(x_k, \varepsilon_k)}\subset B(x_0,
\varepsilon)$ for $k=1,2,\ldots,$ while  $\overline{B(x_k,
\varepsilon_k)}\cap \overline{B(x_p, \varepsilon_p)}=\varnothing$
for $k\ne p,$ $1\leqslant k, p\leqslant j.$

By Theorem~\ref{th1} there is $r_k>0,$ which does not depend on $m,$
such that
$$B_h(f_m(x_k), r_k)\subset f_m(B(x_k, \varepsilon_k))\,,\ldots \,,k=1,2,\ldots j\,,\quad m=1,2,\ldots\,.$$
The latter implies that, for $m\geqslant m_0,$ some $m_0\in {\Bbb
N}$ and $r_{\varepsilon}>0$
\begin{equation}\label{eq1_B}
B_h(f(x_0), r_{\varepsilon})\subset f_m(B(x_k,
\varepsilon_k))\,,\ldots \,,k=1,2,\ldots j\,,\quad m\geqslant m_0\,.
\end{equation}
The relation~(\ref{eq1_B}) implies that there are $a_k\in B(x_k,
\varepsilon_k)$ such that $f_m(a_k)=f(x_0)=y,$ $k=1,2,\ldots, j.$
The latter means that $N(y, f_m, B(x_0, \varepsilon))\geqslant j>q$
that contradicts with~(\ref{eq7_A}). The obtained contradiction
proves the discreteness of $f,$ that finishes the proof.~$\Box$

\medskip
The following statement holds, cf.~\cite{Cr} and \cite{ST}.

\medskip
\begin{theorem}\label{th3}
{\it\, Let $D$ be a domain in ${\Bbb R}^n,$ $n\geqslant 2,$ let
$n-1<p\leqslant n,$ and let $\Phi\colon\overline{{\Bbb
R}^{+}}\rightarrow\overline{{\Bbb R}^{+}}$ be a non-decreasing
function. Besides that, let $f_j:D\rightarrow {\Bbb R}^n,$
$n\geqslant 2,$ $j=1,2,\ldots,$ be a sequence of open discrete
mappings satisfying the conditions~(\ref{eq2*A})--(\ref{eqA2}) with
some function $Q=Q_j:{\Bbb R}^n\rightarrow[0, \infty]$ at any point
$y_0\in \overline{{\Bbb R}^n}$ for which~(\ref{e3.3.1}) holds with
$Q_{f_j}:=Q_j$ and converging to some mapping $f:D\rightarrow
\overline{{\Bbb R}^n}$ as $j\rightarrow\infty$ locally uniformly in
$D$ with respect to the chordal metric $h.$ Assume that the
condition~(\ref{eq2A}) holds. Then either $f$ is a constant in
$\overline{{\Bbb R}^n}$, or $f$ is open discrete mapping
$f:D\rightarrow{\Bbb R}^n.$}
\end{theorem}

\medskip
\begin{proof} The proof is based on Theorem~\ref{th1} and completely
similar to the first part of the proof of Theorem~1.2 in
\cite{Sev$_2$}.~$\Box$
\end{proof}

\section{Connection with Orlicz-Sobolev classes}

The definitions of Sobolev and Orlicz-Sobolev classes used below can
be found in~\cite{Sev$_1$} and are therefore omitted. We also do not
provide many of the notations used below.

\medskip
If $\rho:\,{\Bbb R}^n\rightarrow [0,\infty]$ is a Borel function,
the integral of $\rho$ over $S$ is defined as
\begin{equation}\label{eq3_A}
\int\limits_S \rho^k\, d\mathcal A=\int\limits_{{\Bbb R}^n}
\rho^k(y) N(S, y)\,d\mathcal H^k y. \end{equation}

Let $\Gamma$ be a family of $k$-dimensional surfaces $S$. A Borel
function $\rho\colon {\Bbb R}^n\rightarrow\overline{{\Bbb R}^+}$ is
said to be {\it admissible} for $\Gamma$ (briefly: $\rho\in{\rm
adm}\Gamma$) if
\begin{equation}\label{eq8.2.6}\int\limits_S\rho^k\,d{\mathcal{A}}\geqslant
1\end{equation}
for every surface $S\in\Gamma,$ where the integral on the left-hand
side of~(\ref{eq8.2.6}) is defined by relation~(\ref{eq3_A}). We say
that a property $P$ holds for almost every $k$-dimensional surface,
if $P$ holds for all surfaces except a family of zero $p$-modulus.

The following important information concerning the capacity of a
pair of sets relative to a domain can be found in Ziemer's paper
\cite{Zi$_1$}. Let $G$ be a bounded domain in ${\Bbb R}^n$, and let
$C_{0}, C_{1}$ be nonintersecting compact subsets of the closure of
$G$. Put $R=G \setminus (C_{0} \cup C_{1})$ and $R^{\,*}=R \cup
C_{0}\cup C_{1}$, then the {\it $p$-capacity of the pair $C_{0},
C_{1}$ relative to the closure of~$G$} is defined to be the quantity
$C_p[G, C_{0}, C_{1}] = \inf \int\limits_{R} | \nabla u|^{p}\,
dm(x)$, where the infimum is taken over all functions $u$ continuous
on $R^{\,*}$, $u\in ACL(R)$, with $u=1$ on $C_{1}$ and $u=0$ on
$C_{0}$. Such functions will be called admissible for $C_p [G,
C_{0}, C_{1}]$. A set $\sigma \subset {\Bbb R}^n$ is said to {\it
separate $C_{0}$ and $C_{1}$ in $R^{\,*}$} if $\sigma \cap R$ is
closed in $R$ and there exist nonintersecting set $A$ and $B$ open
in $R^{\,*} \setminus \sigma$ and such that $R^{\,*} \setminus
\sigma = A\cup B$, $C_{0}\subset A$, and $C_{1} \subset B$. Let
$\Sigma$ denote the class of all sets separating $C_{0}$ and $C_{1}$
in $R^{\,*}$. Putting $p^{\,\prime} = p/(p-1)$, we introduce the
quantity $\widetilde{M}_{p^{\,\prime}}(\Sigma)=\inf_{\rho\in
\widetilde{\rm adm} \Sigma} \int\limits_{{\Bbb
R}^n}\rho^{\,p^{\,\prime}}\,dm(x),$
where the formula $\rho\in \widetilde{\rm adm}\Sigma$ means that
$\rho$ is a nonnegative Borel function on ${\Bbb R}^n$ such that
\begin{equation} \label{eq13.4.13}
\int\limits_{\sigma \cap R}\rho \,d{\mathcal H}^{n-1} \geqslant
1\quad\forall\, \sigma \in \Sigma.
\end{equation}
Observe that, by Ziemer's result,
\begin{equation}\label{eq3_4}
\widetilde{M}_{p^{\,\prime}}(\Sigma)=C_p[G , C_{0} ,
C_{1}]^{\,-1/(p-1)},
\end{equation}
see \cite[Theorem~3.13]{Zi$_1$} for $p=n$ and \cite[p.~50]{Zi$_2$}
for $1<p<\infty$. We also observe that, by a result of Hesse,
\begin{equation}\label{eq4_4}
M_p(\Gamma(E, F, D))= C_p[D, E, F]
\end{equation}
under the condition $(E\cup F)\cap\partial D=\varnothing$
see~\cite[Theorem~5.5]{Hes}. Shlyk has proved that the condition
$(E\cup F)\cap\partial D=\varnothing$ may be removed,
see~\cite[Theorem~1]{Shl}.

The following class of mappings is a generalization of
quasiconformal mappings in the sense of the ring definition by
Gehring (see, e.g., \cite[Chapter~9]{MRSY}). Let $p\geqslant 1,$ let
$D$ and $D^{\,\prime}$ be domains in $\overline{{\Bbb R}^n},$
$n\geqslant 2,$ $x_0\in{\Bbb R}^n$ and let
$Q:D\rightarrow(0,\infty)$ be a Lebesgue measurable function. We say
that $f:D\rightarrow D^{\,\prime}$ is a {\it lower $Q$-mapping at
$x_0$ with respect to the $p$-modulus,} whenever
\begin{equation}\label{eq1A_1}
M_p(f(\Sigma_{\varepsilon}))\geqslant \inf\limits_{\rho\in{\rm
ext}_p\,{\rm adm}\,\Sigma_{\varepsilon}}\int\limits_{D\cap A(x_0,
\varepsilon, r_0)}\frac{\rho^p(x)}{Q(x)}\,dm(x)
\end{equation}
for each ring $A(x_0, \varepsilon, r_0)=\{x\in {\Bbb R}^n:
\varepsilon<|x-x_0|<r_0\},$ $r_0\in(0, d_0),$ $d_0=\sup\limits_{x\in
D}|x-x_0|,$
where $\Sigma_{\varepsilon}$ denotes the family of all intersections
of spheres $S(x_0, r)$ with domain $D,$ $r\in (\varepsilon, r_0).$
The following statement facilitates the verification of the infinite
series of inequalities in (\ref{eq1A_1}) and may be established
similarly to the proof of \cite[Theorem~9.2]{MRSY}.

\medskip
\begin{lemma}\label{lemma4}{\,
Let $D,$ $D^{\,\prime}\subset\overline{{\Bbb R}^n},$
$x_0\in\overline{D}\setminus\{\infty\}$ and let
$Q:D\rightarrow(0,\infty)$ be a Lebesgue measurable function. Then a
mapping $f:D\rightarrow D^{\,\prime}$ is a lower $Q$-mapping with
respect to the $p$-modulus at $x_0,$ $p>n-1,$ if and only if
%
$M_p(f(\Sigma_{\varepsilon}))\geqslant\int\limits_{\varepsilon}^{r_0}
\frac{dr}{||\,Q||_{s}(r)}\quad\forall\ \varepsilon\in(0,r_0)\,,\
r_0\in(0,d_0),$ $d_0=\sup\limits_{x\in D}|x-x_0|,$
%
$s=\frac{n-1}{p-n+1},$ where, as above, $\Sigma_{\varepsilon}$
denotes the family of all intersections of spheres $S(x_0, r)$ with
the domain $D,$ $r\in (\varepsilon, r_0),$
$ \Vert
Q\Vert_{s}(r)=\left(\int\limits_{D(x_0,r)}Q^{s}(x)\,d{\mathcal{A}}\right)^{\frac{1}{s}}$
-- $L_{s}$-norm of function $Q$ over sphere $D(x_0,r)=\{x\in D:
|x-x_0|=r\}=D\cap S(x_0,r)$.}
\end{lemma}

\medskip
We define for any $x\in D$ and fixed $p,q$, $p,q\geqslant 1$
\begin{equation}\label{eq15.1}
K_{I, q}(x,f)\quad =\quad\left\{
\begin{array}{rr}
\frac{|J(x,f)|}{{l\left(f^{\,\prime}(x)\right)}^q}, & J(x,f)\ne 0,\\
1,  &  f^{\,\prime}(x)=0, \\
\infty, & {\rm otherwise}
\end{array}
\right.\,.\end{equation}

The following statement was first proved for homeomorphisms and
$x_0\in\overline{D}$ in \cite[Theorem~2.1]{KR},
cf.~\cite[Lemma~2.3]{PSS}. The proof of this statement for the case
$x_0\in {\Bbb R}^n$ does not differ from the above case
$x_0\in\overline{D}.$

\medskip
\begin{lemma}{}\label{thOS4.1} {\it\,
Let $D$ be a domain in ${\Bbb R}^n,$ $n\geqslant 3,$
$\varphi:(0,\infty)\rightarrow (0,\infty)$ be a non-decreasing
function satisfying the condition
$
\int\limits_{1}^{\infty}\left(\frac{t}{\varphi(t)}\right)^
{\frac{1}{n-2}}dt<\infty\,.
$
If $n\geqslant 3$ and $p>n-1,$ then every open discrete mapping
$f:D\rightarrow {\Bbb R}^n$ with finite distortion of the class
$W^{1,\varphi}_{loc}$ such that $N(f, D)<\infty$ is a lower
$Q$-mapping with respect to $p$-modulus at each point $x_0\in{\Bbb
R}^n$ for
$Q(x)=N(f, D)\cdot K^{\frac{p-n+1}{n-1}}_{I, \alpha}(x, f),$
$\alpha:=\frac{p}{p-n+1},$ where the inner dilatation
$K_{I,\alpha}(x, f)$ of $f$ at $x$ is of order $\alpha$ is defined
by the relation (\ref{eq15.1}), and the multiplicity $N(f, D)$ is
defined by the second relation in (\ref{eq1G}).}
\end{lemma}

\medskip
In what follows we will need the following auxiliary assertion (see,
e.g., \cite[Lemma~7.4, Ch.~7]{MRSY} for $p=n$ and
\cite[Lemma~2.2]{Sal} for $p\ne n.$

\medskip
\begin{proposition}\label{pr1A}
{\,\it Let $x_0 \in {\Bbb R}^n,$ $Q(x)$ be a Lebesgue measurable
function, $Q:{\Bbb R}^n\rightarrow [0, \infty],$ $Q\in
L_{loc}^1({\Bbb R}^n).$ We set $A:=A(x_0, r_1, r_2)=\{ x\,\in\,{\Bbb
R}^n : r_1<|x-x_0|<r_2\}$ and
$\eta_0(r)=\frac{1}{Ir^{\frac{n-1}{p-1}}q_{x_0}^{\frac{1}{p-1}}(r)},$
where $I:=I=I(x_0,r_1,r_2)=\int\limits_{r_1}^{r_2}\
\frac{dr}{r^{\frac{n-1}{p-1}}q_{x_0}^{\frac{1}{p-1}}(r)}$ and
$q_{x_0}(r):=\frac{1}{\omega_{n-1}r^{n-1}}\int\limits_{|x-x_0|=r}Q(x)\,d{\mathcal
H}^{n-1}$ is the integral average of the function $Q$ over the
sphere $S(x_0, r).$ Then
\begin{equation*}\label{eq10A_1}
\frac{\omega_{n-1}}{I^{p-1}}=\int\limits_{A} Q(x)\cdot
\eta_0^p(|x-x_0|)\ dm(x)\leqslant\int\limits_{A} Q(x)\cdot
\eta^p(|x-x_0|)\ dm(x)
\end{equation*}
for any Lebesgue measurable function $\eta :(r_1,r_2)\rightarrow
[0,\infty]$ such that
$\int\limits_{r_1}^{r_2}\eta(r)dr=1. $ }
\end{proposition}

The following statement was proved for some cases earlier, see e.g.
\cite[Proposition~3]{KR}.

\medskip
\begin{theorem}\label{th4_2}
{\,\it Let $x_0\in {\Bbb R}^n,$ let $f:D\rightarrow {\Bbb R}^n$ be a
bounded lower $Q$-homeomorphism with respect to $p$-modulus in a
domain $D\subset{\Bbb R}^n,$ $Q\in L_{loc}^{\frac{n-1}{p-n+1}}({\Bbb
R}^n),$ $n-1<p\leqslant n,$ and $\alpha:=\frac{p}{p-n+1}.$ Then, for
every $0<\varepsilon<\varepsilon_0<\infty$ and any compact sets
$C_2\subset D\setminus B(x_0, \varepsilon_0)$ and $C_1\subset
\overline{B(x_0, \varepsilon)}\cap D$ the inequality
$
M_{\alpha}(f(\Gamma(C_1, C_2, D)))\leqslant \int\limits_{A(x_0,
\varepsilon, \varepsilon_1)}Q^{\frac{n-1}{p-n+1}}(x)
\eta^{\alpha}(|x-x_0|)dm(x)
$
holds, where $A(x_0, \varepsilon, \varepsilon_1)=\{x\in {\Bbb R}^n:
\varepsilon<|x-x_0|<\varepsilon_1\}$ and $\eta: (\varepsilon,
\varepsilon_1)\rightarrow [0,\infty]$ is an arbitrary Lebesgue
measurable function such that
$
\int\limits_{\varepsilon}^{\varepsilon_1}\eta(r)dr=1\,.
$
}
\end{theorem}

\begin{proof} Note that,
the set $\sigma_r:= f(S(x_0, r)\cap D)$ is closed in $f(D).$ In
addition, note that $\sigma_r$ separates $f(C_1)$ from $f(C_2)$ for
$r\in (\varepsilon, \varepsilon_0)$ in $f(D),$ since
$$f(C_1)\subset f(B(x_0, r)\cap D):=A,\quad
f(C_2)\subset f(D)\setminus \overline{f(B(x_0, r)\cap D)}:=B\,,$$
$A$ and $B$ are open in $f(D)$ and $f(D)=A\cup \sigma_r\cup B.$ Let
$\Sigma_{\varepsilon}$ be the family of all sets separating $f(C_1)$
from $f(C_2)$ in $f(D).$ Let $\rho^{n-1}\in \widetilde{{\rm
adm}}\bigcup\limits_{r\in (\varepsilon, \varepsilon_0)} f(S(x_0,
r)\cap D)$ in the sense of the relation (\ref{eq13.4.13}). Then also
$\rho\in {\rm adm}\bigcup\limits_{r\in (\varepsilon, \varepsilon_0)}
f(S(x_0, r)\cap D)$ in the sense of the relation (\ref{eq8.2.6}) for
$k=n-1.$ Therefore, since
$\widetilde{M}_{q}(\Sigma_{\varepsilon})\geqslant
M_{q(n-1)}(\Sigma_{\varepsilon})$ for arbitrary $q\geqslant 1,$ we
have that
\begin{equation}\label{eq5A_1}
\widetilde{M}_{p/(n-1)}(\Sigma_{\varepsilon})\geqslant
M_p\left(\bigcup\limits_{r\in (\varepsilon, \varepsilon_1)} f(S(x_0,
r)\cap D)\right)\,.
\end{equation}
However, due to (\ref{eq3_4}) and (\ref{eq4_4}),
$
\widetilde{M}_{p/(n-1)}(\Sigma_{\varepsilon})=\frac{1}{(M_{\alpha}(\Gamma(f(C_1),
f(C_2), f(D))))^{1/(\alpha-1)}}\,.
$
By Lemma~\ref{lemma4}
$$M_{p}\left(\bigcup\limits_{r\in (\varepsilon, \varepsilon_1)} f(S(x_0,
r)\cap D)\right)\geqslant
$$
\begin{equation}\label{eq8B_1}
\geqslant \int\limits_{\varepsilon}^{\varepsilon_1}
\frac{dr}{\Vert\,Q\Vert_{s}(r)}=
\int\limits_{\varepsilon}^{\varepsilon_1}
\frac{dt}{\omega^{\frac{p-n+1}{n-1}}_{n-1}
t^{\frac{n-1}{\alpha-1}}\widetilde{q}_{x_0}^{\,\frac{1}{\alpha-1}}(t)}\quad\forall\,\,
i\in {\Bbb N}\,, s=\frac{n-1}{p-n+1}\,,\end{equation} where
$\Vert
Q\Vert_{s}(r)=\left(\int\limits_{D(x_0,r)}Q^{s}(x)\,d{\mathcal{A}}\right)^{\frac{1}{s}}$
is $L_{s}$-norm of the function $Q$ over the sphere $S(x_0,r)\cap D$
and
$\widetilde{q}_{x_0}(r):=\frac{1}{\omega_{n-1}r^{n-1}}\int\limits_{|x-x_0|=r}Q^s(x)\,d{\mathcal
H}^{n-1}.$ Then from (\ref{eq5A_1})--(\ref{eq8B_1}) it follows that
\begin{equation}\label{eq9C}
M_{\alpha}(\Gamma(f(C_1), f(C_2), f(D)))\leqslant
\frac{\omega_{n-1}}{I^{\alpha-1}}\,,
\end{equation}
where $I=\int\limits_{\varepsilon}^{\varepsilon_1}\
\frac{dr}{r^{\frac{n-1}{\alpha-1}}\widetilde{q}_{x_0}^{\frac{1}{\alpha-1}}(r)}.$
Note that $f(\Gamma(C_1,C_2, D))\subset \Gamma(f(C_1), f(C_2),
f(D)),$ so that from (\ref{eq9C}) it follows that
$
M_{\alpha}(f(\Gamma(C_1,C_2, D)))\leqslant
\frac{\omega_{n-1}}{I^{\alpha-1}}\,.
$
The proof is completed by applying Proposition~\ref{pr1A}.~$\Box$
\end{proof}

\medskip
By Lemma~\ref{thOS4.1} and Theorem~\ref{th4_2} we obtain the
following.

\medskip
\begin{theorem}\label{th5} {\it\,
Let $D$ be a domain in ${\Bbb R}^n,$ $n\geqslant 3,$
$\varphi:(0,\infty)\rightarrow (0,\infty)$ be a non-decreasing
function satisfying condition
\begin{equation}\label{eqOS3.0a}
\int\limits_{1}^{\infty}\left(\frac{t}{\varphi(t)}\right)^
{\frac{1}{n-2}}dt<\infty\,.
\end{equation}
If $0<\varepsilon<\varepsilon_0<\infty,$ $n\geqslant 3$ and
$n-1<\alpha\leqslant n,$ then every bounded homeomorphism
$f:D\rightarrow {\Bbb R}^n$ with finite distortion in
$W^{1,\varphi}_{loc}$ with $K_{I, \alpha}(x, f)\in L^1_{\rm loc}(D)$
satisfies the relation
$M_{\alpha}(f(\Gamma(C_1, C_2, D)))\leqslant \int\limits_{A(x_0,
\varepsilon, \varepsilon_1)}Q(x)\eta^{\alpha}(|x-x_0|)\,dm(x)$
for $Q(x)=K_{I, \alpha}(x, f),$ $A(x_0, \varepsilon,
\varepsilon_1)=\{x\in {\Bbb R}^n:
\varepsilon<|x-x_0|<\varepsilon_1\},$ for any compact sets
$C_2\subset D\setminus B(x_0, \varepsilon_0)$ and $C_1\subset
\overline{B(x_0, \varepsilon)}\cap D$ and for an arbitrary Lebesgue
measurable function $\eta:(\varepsilon, \varepsilon_1)\rightarrow
[0,\infty]$ such that
$\int\limits_{\varepsilon}^{\varepsilon_1}\eta(r)dr=~1\,.$ }
\end{theorem}

\medskip
Let now $f:D\rightarrow D^{\,\prime},$ $x_0\in D$ and $y_0=f(x_0).$
Assume that, $f$ is a homeomorphism and denote by $g:=f^{\,-1}.$
Observe that $f(\Gamma(C_1, C_2, D))=f(\Gamma(S(x_0, \varepsilon)),
S(x_0, \varepsilon_0), D)$
and, consequently,
$\Gamma(C_1, C_2, D)=\Gamma_g(x_0, \varepsilon, \varepsilon_0).$

\medskip
Given $p\geqslant 1,$ a non-decreasing function
$\Phi\colon\overline{{\Bbb R}^{+}}\rightarrow\overline{{\Bbb
R}^{+}},$ $a, b\in D,$ $a\ne b,$ $\delta>0$ we denote by
$\frak{O}^{\Phi, p}_{a, b, \delta}(D)$ the family of all
homeomorphisms $f:D\rightarrow {\Bbb R}^n,$ $n\geqslant 2,$ such
that its inverses $g:f(D)\rightarrow D$ are of the class
$W^{1,\varphi}_{loc}$ and have finite distortion such that
$$
\int\limits_{{\Bbb R}^n}\Phi(K_{I, p}(y,
g))\cdot\frac{dm(y)}{(1+|y|^2)^n}<\infty
$$
such that $h(f(a), f(b))\geqslant \delta.$ The following statement
holds.

\medskip
\begin{theorem}\label{th4}
{\it Let $n\geqslant 2,$ $p\in (n-1, n],$ and let $D$ be a bounded
domain in ${\Bbb R}^n.$ Assume that, the corresponding family
$\frak{O}^{\Phi, p}_{a, b, \delta}(D)$ is equicontinuous at $a$ and
$b,$ in addition, $\Phi:\overline{{\Bbb R^{+}}}\rightarrow
\overline{{\Bbb R^{+}}}$ is an increasing convex function that
satisfies the condition
$
\int\limits_{\delta}^{\infty}\frac{d\tau}{\tau(\Phi^{\,-1}(\tau))^{\frac{1}{p-1}}}=\infty
$
for some~$\delta>\Phi(0).$ Assume that the relation~(\ref{eqOS3.0a})
holds.

\medskip
Then for every compactum $K$ in $D$ and for every
$0<\varepsilon<{\rm dist}\,(K,
\partial D)$ there exists
$r_0=r_0(\varepsilon, K)>0$ which does not depend on $f$ such that
$f(B(x_0, \varepsilon))\supset B_h(f(x_0), r_0)$
for all $f\in\frak{O}^{\Phi, p, \varphi}_{a, b, \delta}(D)$ and all
$x_0\in K,$ where $B_h(f(x_0), r_0)=\{w\in \overline{{\Bbb R}^n}:
h(w, f(x_0))<r_0\}.$}
\end{theorem}

\medskip
\begin{proof} By Theorem~\ref{th5} and comments made before the
formulation of Theorem~\ref{th4}, $\frak{O}^{\Phi, p, \varphi}_{a,
b, \delta}(D)\subset \frak{F}^{\Phi, p}_{a, b, \delta}(D).$ The rest
follows by Theorem~\ref{th1}.~$\Box$
\end{proof}

\medskip
{\bf Statements \& Declarations}

\medskip
{\bf Funding.} The work was supported by the National Research
Foundation of Ukraine (Project ``Analogues of Carath\'{e}odory and
Koebe-Bloch theorems for Orlycz-Sobolev classes'', Project number
2025.02/0010).

\medskip
{\bf Competing Interests.} The authors have no relevant financial or
non-financial interests to disclose

\medskip
{\bf Availability of data and material.} The datasets generated
and/or analysed during the current study are available from the
corresponding author on reasonable request.

\medskip
{\bf Authors' contributions.} All authors contributed to the study
conception and design. All authors read and approved the final
manuscript.


CONTACT INFORMATION

\medskip
{\bf \noindent Evgeny Sevost'yanov} \\
{\bf 1.} Zhytomyr Ivan Franko State University,  \\
40 Bol'shaya Berdichevskaya Str., 10 008  Zhytomyr, UKRAINE \\
{\bf 2.} Institute of Applied Mathematics and Mechanics\\
of NAS of Ukraine, \\
19 Henerala Batyuka Str., 84 116 Slov'yans'k,  UKRAINE\\
esevostyanov2009@gmail.com

\medskip
{\bf \noindent Valery Targonskii} \\
Zhytomyr Ivan Franko State University,  \\
40 Bol'shaya Berdichevskaya Str., 10 008  Zhytomyr, UKRAINE \\
w.targonsk@gmail.com

\medskip
{\bf \noindent Nataliya Ilkevych} \\
Zhytomyr Ivan Franko State University,  \\
40 Bol'shaya Berdichevskaya Str., 10 008  Zhytomyr, UKRAINE \\
Email: ilkevych1980@gmail.com


\begin{thebibliography}{99}
{\small

\bibitem[AFW]{AFW}
{\sc Adamowicz, T., K.~F\"{a}ssler, B.~Warhurst:} A Koebe distortion
theorem for quasiconformal mappings in the Heisenberg group. - Ann.
Mat. Pura Appl. (4) 199:1, 2020, 147–-186.

\bibitem[AS]{AS} {\sc Adamowicz, T. and N.~Shanmugalingam:}
Non-conformal Loewner type estimates for modulus of curve families.
- Ann. Acad. Sci. Fenn. Math. 35, 2010, 609–-626.

\bibitem[CG]{CG} {\sc Carleson, L., T.W.~Gamelin:} Complex dynamics,
Universitext: Tracts in Mathematics. - Springer-Verlag, New York
etc., 1993.

\bibitem[Car]{Car}
{\sc Caraman, P.:} Relations between $p$-capacity and
$p$-module~(I). - Rev. Roum. Math. Pures Appl.~39:6, 1994, 509--553.

\bibitem[Cr]{Cr} {\sc Cristea, M.:} On the Limit Mapping of a Sequence of Open, Discrete
Mappings Satisfying an Inverse Poletsky Modular Inequality. -
Complex Analysis and Operator Theory~19, 2025, Article number~165.

\bibitem[HK]{HK} {\sc Hajlasz, P. and P.~Koskela:} Sobolev met
Poincare. - Mem. Amer. Math. Soc.~145:688, 2000, 1--101.

\bibitem[Hes]{Hes} {\sc Hesse, J.:}
A $p-$extremal length and $p-$capacity equality. - Ark. Mat.~13,
1975, 131--144.

\bibitem[He]{He} {\sc Heinonen, J.:} Lectures on Analysis on metric
spaces. - Springer Science+Business Media, New York, 2001.

\bibitem[KR]{KR} {\sc Kovtonyuk, D.A. and V.I.~Ryazanov:} New modulus estimates in Orlicz-Sobolev
classes. - Annals of the University of Bucharest (mathematical
series)~5(LXIII), 2014, 131--135.


\bibitem[MRV]{MRV} {\sc Martio, O., S. Rickman, and J. V\"{a}is\"{a}l\"{a}:}
Topological and metric properties of quasiregular mappings. - Ann.
Acad. Sci. Fenn. Ser. A1. 488, 1971, 1--31.

\bibitem[MRSY]{MRSY} {\sc Martio, O., V. Ryazanov, U. Srebro, and E. Yakubov:}
Moduli in modern mapping theory. - Springer Science + Business
Media, LLC, New York, 2009.

\bibitem[M]{M} {\sc Mateljevi\'{c}, M.:}
Quasiconformal and quasiregular harmonic analogues of Koebe's
theorem and applications. - Ann. Acad. Sci. Fenn. Math. 32:2, 2007,
301–-315.

\bibitem[Na$_1$]{Na$_1$} {\sc N\"{a}kki, R.:} Boundary behavior
of quasiconformal mappings in $n$-space. - Ann. Acad. Sci. Fenn.
Ser. A.~484, 1970, 1--50.

\bibitem[Na$_2$]{Na$_2$} {\sc N\"{a}kki, R.:} Extension of Loewner's capacity
theorem. - Trans. Amer. Math. Soc. 180, 1973, 229--236.

\bibitem[PSS]{PSS} {\sc Sevost’yanov, E., R.~Salimov, E.~Petrov:} On
the removable of singularities of the Orlicz-Sobolev classes. - J.
Math. Sci.~222:6, 2017, 723--740.

\bibitem[Ra]{Ra} {\sc Rajala, K.:} Bloch's theorem for mappings of
bounded and finite distortion. - Math. Ann. 339:2, 2007, 445–-460.

\bibitem[Re]{Re} {\sc Reshetnyak, Yu.G.:} Space Mappings with Bounded Distortion.
Transl. of Math. Monographs \textbf{73}, AMS, 1989.

\bibitem[Ri]{Ri} {\sc Rickman, S.:} Quasiregular mappings.
Springer-Verlag, Berlin, 1993.

\bibitem[RS]{RS} {\sc Ryazanov, V., E.~Sevost'yanov:}
Equicontinuity of mappings quasiconformal in the mean. - Ann. Acad.
Sci. Fenn. 36, 2011, 231-244.

\bibitem[Sal]{Sal} {\sc Salimov, R.R.:} Estimation of the measure of the image of the ball. - Sib.
Math. J.~53:4, 2012, 739--747.

\bibitem[Sev$_1$]{Sev$_1$} {\sc Sevost'yanov, E.:}
Mappings with Direct and Inverse Poletsky Inequalities. Developments
in Mathematics (DEVM, volume 78). - Springer Nature Switzerland AG,
Cham, 2023.

\bibitem[Sev$_2$]{Sev$_2$} {\sc Sevost'yanov, E.:} On global behavior of mappings
with integral constraints. - Analysis and Mathematical Physics~12:3,
2022, Article number~76.

\bibitem[SKN]{SKN} {\sc Sevost’yanov, E., Z.~Kovba, G.~Nosal:}
On the lower bounds of the $p$-modulus of families. - Journal of
Mathematical Sciences. https://doi.org/10.1007/s10958-025-07789-y .

\bibitem[ST]{ST} {\sc Sevost'yanov, V.A.~Targonskii:}
An analogue of Koebe's theorem and the openness of a limit map in
one class. - Analysis and Mathematical Physics~15:3, 2025, Article
number~59.

\bibitem[SSD]{SSD} {\sc Sevost'yanov, E.A., S.O.~Skvortsov, O.P.~Dovhopiatyi:}
On nonhomeomorphic mappings with the inverse Poletsky inequality. -
Journal of Mathematical Sciences 252:4, 2021, 541--557.

\bibitem[Shl]{Shl} {\sc Shlyk, V.A.:} The equality between $p$-capacity and
$p$-modulus. - Siberian Mathematical Journal~34:6, 1993, 1196--1200.

\bibitem[Va]{Va} {\sc V\"{a}is\"{a}l\"{a}, J.:} Lectures
on $n$-dimensional quasiconformal mappings. - Lecture Notes in
Math. 229, Springer-Verlag, Berlin etc., 1971.

\bibitem[Zi$_1$]{Zi$_1$} {\sc Ziemer, W.P.:} Extremal length and conformal capacity. - Trans. Amer.
Math. Soc.~126:3, 1967, 460--473.

\bibitem[Zi$_2$]{Zi$_2$} {\sc Ziemer, W.P.:}
Extremal length and $p-$capacity. - Michigan Math. J.~16, 1969,
43--51.}


\end{thebibliography}
\end{document}